\documentclass[psamsfonts]{amsart}
\usepackage{amssymb,amsfonts}
\usepackage{enumerate}
\usepackage{mathrsfs}
\usepackage{latexsym}
\usepackage{url}
\usepackage{mathtools}

\newtheorem{thm}{Theorem}[section]
\newtheorem{cor}[thm]{Corollary}
\newtheorem{prop}[thm]{Proposition}

\theoremstyle{definition}
\newtheorem{defn}[thm]{Definition}
\newtheorem{defns}[thm]{Definitions}

\newtheorem{exmp}[thm]{Example}

\newtheorem{notn}[thm]{Notation}

\theoremstyle{remark}
\newtheorem{rem}[thm]{Remark}

\newcommand{\finseq}[2]{s^F_{#1}(#2)}
\newcommand{\seq}[2]{s_{#1}(#2)}
\newcommand{\finseqeq}[2]{e^F_{#1}(#2)}
\newcommand{\indec}[2]{z^F_{#1}(#2)}
\newcommand{\seqeq}[2]{e_{#1}(#2)}
\newcommand{\pfin}[1]{\wp_{\mathrm{fin}}(#1)}
\newcommand{\pfinz}[1]{\wp'_{\mathrm{fin}}(#1)}
\newcommand{\pfinn}{\wp_{\mathrm{fin}}}
\newcommand{\pfinnz}{\wp'_{\mathrm{fin}}}
\DeclareMathOperator*{\dom}{dom}
\DeclareMathOperator*{\im}{im}

\makeatletter
\let\c@equation\c@thm
\makeatother
\numberwithin{equation}{section}

\author{Harry Altman}
\title{Bounding finite-image sequences of length $\omega^k$}
\date{December 1, 2026}

\begin{document}
\begin{abstract}
Given a well-quasi-order $X$ and an ordinal $\alpha$, the set
$\finseq{\alpha}{X}$ of transfinite sequences on $X$ with length less than
$\alpha$ and with finite image is also a well-quasi-order, as proven by
Nash-Williams. \cite{NW}  Before Nash-Williams proved it for general $\alpha$,
however, it was proven for $\alpha<\omega^\omega$ by Erd\H{o}s and Rado.
\cite{ER}  In this paper, we revisit Erd\H{o}s and Rado's proof and improve upon
it, using it to obtain upper bounds on the maximum linearization of
$\finseq{\omega^k}{X}$ in terms of $k$ and $o(X)$, where $o(X)$ denotes the
maximum linearization of $X$.  We show that, for fixed $k$,
$o(\finseq{\omega^k}{X})$ is bounded above by a function which can roughly be
described as $(k+1)$-times exponential in $o(X)$.  We also show that, for $k\le
2$, this bound is not far from tight.
\end{abstract}

\maketitle

\section{Introduction}

Suppose $X$ is a well-quasi-order.  Then there is a natural way of making $X^*$,
the set of finite-length strings on $X$, into a well-quasi-order as well; this
is Higman's Lemma. \cite{Higman}  In 1965 C.~Nash-Williams extended this from
finite strings to transfinite sequences, with the qualification that one must
restrict attention to sequences with finite image; we only allow a sequence if
that sequence uses only finitely many symbols from $X$. \cite{NW} (Nash-Williams
also showed that this restriction can be removed if $X$ is a better-quasi-order
\cite{bqo}, but we will not consider this here.)

Given $X$ and an ordinal $\alpha$, we will define $\finseq{\alpha}{X}$ to be the
set of finite-image sequences on $X$ with length less than $\alpha$.  (So
$\finseq{\omega}{X}$ is just $X^*$, since a string of finite length must
certainly use only finitely many symbols.)  The question then arises: Can we
determine a nontrivial upper bound on the type of $\finseq{\alpha}{X}$, in terms
of $\alpha$ and the type of $X$?  (In the cases where we do not want a
finiteness restriction, we will write $\seq{\alpha}{X}$.)

Here by the \emph{type} of a well-quasi-order $X$ we mean the largest order type
of a linearization of $X$ (after quotienting out by equivalences); we denote it
$o(X)$.  This quantity was proven to exist by de Jongh and Parikh \cite{DJP},
who also showed that $o(X)$ can be characterized inductively as the smallest
ordinal greater than $o(Y)$ for any proper lower set $Y$ of $X$.  The theory has
been rediscovered several times since then; indeed the term ``type'' comes from
Kriz and Thomas. \cite{KT}  This quantity is also known by other names, such as
the maximum linearization.

Of course one could write down a trivial bound based on cardinality.  But this
is uninteresting.  The question then is if one can do better; and, moreover, if
one can come up with an upper bound that is at least reasonably tight.  Schmidt
claimed a nontrivial upper bound \cite{Schmidt}, but the proof has a hole which
has never, to this author's awareness, been repaired.\footnote{The gap may be
found on line 7 of page 34 of \cite{Schmidt}; the second occurrence of
$s_\gamma^F$ should simply be an $s_\gamma$, as the sequence need not have
finite image, but without this finite-image condition, the proof cannot
continue.  Of course, Theorem~6, which is what's being proven at that point, is
true regardless; but the argument is later referred back to to prove Lemma~9 as
part of the proof of the upper bound, and the argument is not valid, so Lemma~9
may not hold.  Thanks to Andreas Weiermann for pointing this out to me.}  So
this problem remains open.

Since the general problem seems difficult, we will restrict our attention to an
easier special case, the case of small $\alpha$.  It's already been mentioned
above that Higman's Lemma, the case of $\alpha=\omega$, was proven before the
full version of Nash-Williams's theorem, and the problem of computing $o(X^*)$
from $o(X)$ was solved by de Jongh, Parikh, and Schmidt \cite{DJP,Schmidt}.  But
there's more history to this problem inbetween these two endpoints.  Before
Nash-Williams's theorem was proved, Erd\H{o}s and Rado proved the special case
of $\alpha<\omega^\omega$ \cite{ER}; before that, Rado proved the case of
$\alpha=\omega^3$. \cite{omega3}  (A note, the result of Erd\H{o}s and Rado is
sometimes misstated as having been for $\alpha=\omega^\omega$, as in
\cite{basic}; but in fact it was only for $\alpha<\omega^\omega$.  A subsequent
paper by Chopra and Pakhomov \cite{trees} will show how to obtain a bound on the
type in the case $\alpha=\omega^\omega$.)

In this paper we will reexamine this proof of Erd\H{o}s and Rado.  This proof
was written before the notion of $o(X)$ was ever defined, but we will see that
we can extract from it an upper bound on $o(\finseq{\omega^k}{X})$.  Actually,
we will improve somewhat on this proof; the bound we will derive will be tighter
than the one that would be obtained from a direct proof-mine.

Specifically:
\begin{thm}
\label{upper-intro}
For any fixed $k$ which is finite and nonzero, the type of
$\finseq{\omega^k}{X}$ is bounded above by a function which is $(k+1)$-times
exponential in $o(X)$ (in an appropriate sense to be made clear below).  (See
Theorem~\ref{upper} for a precise version of this theorem.)
\end{thm}

By contrast, a direct proof-mine of Erd\H{o}s and Rado's paper would result in a
tower of $2k$ exponentials, rather than $k+1$.

Note, by the way, that this bound is much smaller than Schmidt's claimed bound,
so Schmidt's claimed bound (which does not seem to have been intended to be
tight in any way) is at least true in this case.

Unfortunately, it does not seem possible to extract a bound from Nash-Williams's
proof in the same way.  So the problem of finding a bound for the general case,
or even just for $o(\finseq{\omega^\omega}{X})$, remains open.

We will address the more general case of $\alpha<\omega^\omega$ in
Section~\ref{cnf}.

In Section~\ref{seclower}, we will also examine the question of lower bounds;
unfortunately, we are not able to show that our bound is reasonably tight in
general, but we will at least show this when $n=2$ (the case of $n=1$ being
already known due to the work of de Jongh, Parikh, and Schmidt
\cite{DJP,Schmidt}).

More specifically:
\begin{thm}
There is a triply-exponential function $f(\beta)$ such that, for any $\beta$,
there is some well-quasi-ordering $X$ with $o(X)=\beta$ and
$o(\finseq{\omega^2}{X})\ge f(\beta)$.
\end{thm}

Hopefully in the future this can be extended to larger $k$.

\section{Preliminaries}
\label{prelim}

We first recall some preliminaries about the behavior of $o(X)$.
Much of the theory is often stated for well partial orders, with antisymmetry,
rather than well-quasi-orders, without it, but of course this difference is
immaterial; all statements should be taken as being up to equivalence in the
relevant quasi-orders.  That is to say, in a sense we are really talking about
well partial orders, but we use well-quasi-orders here for convenience.

If $Y$ is a well-quasi-order and $X$ is a subset of $Y$, then $o(X)\le o(Y)$, so
more generally if $X$ embeds in $Y$, then $o(X)\le o(Y)$.  Moreover, if $X$ is a
lower set in $Y$ and is also a proper subset of $Y$, then $o(X)<o(Y)$; indeed,
$o(Y)$ is equal to the smallest ordinal greater than every $o(X)$, where $X$
ranges over all proper lower subsets of $Y$.  \cite{DJP} Relatedly, if there is
a monotonic surjection from $X$ onto $Y$, then $o(Y)\le o(X)$.  (This is because
$o(X)$ can also be characterized as the largest ordinal that $X$ has a monotonic
surjection to; if $X$ has a monotonic surjection onto $Y$, then any ordinal that
$Y$ surjects onto, $X$ will surject onto.)  In particular, if $X$ and $Y$ are
quasi-orders on the same set with the order $Y$ being an extension of the order
$X$, then $o(Y)\le o(X)$.  \cite{DJP}

Given two well-quasi-orders $X$ and $Y$, the disjoint union $X\amalg Y$ and
Cartesian product $X\times Y$ are well-quasi-orders with $o(X\amalg
Y)=o(X)\oplus o(Y)$ and $o(X\times Y)=o(X)\otimes o(Y)$, where $\oplus$ and
$\otimes$ represent the natural (or ``Hessenberg'' \cite[pp. 73--81]{hdf})
addition and multiplication on the ordinals, distinct from their ordinary
addition and multiplication.  \cite{DJP} These are also the addition and
multiplication on the ordinals that one gets from the surreals. \cite{ONAG}

It's also worth noting that $o(X)$ is a successor ordinal if and only if $X$
contains a maximal element (after equivalences); in this case one may remove
that maximal element (and its equivalence class) to obtain $X'$ with
$o(X)=o(X')+1$. \cite{DJP}

Our central object of study here is $\finseq{\alpha}{X}$:

\begin{defn}
Given a set $X$ and an ordinal $\alpha$, define $\seq{\alpha}{X}$ to be the set
of (possibly transfinite) sequences of length less than $\alpha$ taking values
in $X$.  (That is, functions $s:\beta\to X$, where $\beta$ is an ordinal less
than $\alpha$.)  Moreover, define $\finseq{\alpha}{X}$ to be the subset of
$\seq{\alpha}{X}$ consisting of those sequences with finite image.
\end{defn}

In the case that $X$ is equipped with a well-quasi-order, one can use this to
define a quasi-order on $\finseq{\alpha}{X}$, which Nash-Williams showed
\cite{NW} is a well-quasi-order:

\begin{defn}
Given a well-quasi-order $X$ and an ordinal $\alpha$, and given
$s,t\in\finseq{\alpha}{X}$, define $s\le t$ if there exists a strictly
increasing function $i:\dom(s)\to\dom(t)$ such that, for all
$\alpha'\in\dom(s)$, one has $s(\alpha')\le t(i(\alpha'))$.
\end{defn}

As noted earlier, $\finseq{\omega}{X}$ is just $X^*$, the set of strings on $X$.
In this case, we have:

\begin{thm}[de Jongh, Parikh, Schmidt \cite{DJP,Schmidt}]
\label{higman-type}
Let $X$ be a well-quasi-order.  Then:
\begin{itemize}
\item If $o(X)=0$, then $o(X^*)=1$.
\item If $o(X)$ is nonzero and finite, then $o(X^*)=\omega^{\omega^{o(X)-1}}$.
\item If $o(X)$ can be written as $\varepsilon+k$, where $\varepsilon$ is an
epsilon number and $k$ is finite, then $o(X^*)=\omega^{\omega^{o(X)+1}}$.
\item Otherwise, $o(X^*)=\omega^{\omega^{o(X)}}$.
\end{itemize}
\end{thm}

It's worth observing here that in all of the constructions above -- disjoint
union, Cartesian product, strings on an alphabet -- the type of the output
depended on only on the type of the input.  However this will not hold in
general; in particular it will not hold for $\finseq{\alpha}{X}$ once
$\alpha>\omega$.  For instance, as observed by Schmidt \cite{Schmidt}, one has
$o(\finseq{\omega+1}{2})=\omega^\omega 2+1$ but $o(\finseq{\omega+1}{1\amalg
1})=\omega^\omega 3+1$.  Thus in general one can only put bounds on the result,
rather than determine the result in terms of $o(X)$ alone.

Finally, given an well-quasi-order $X$, we'll also consider a well-quasi-order
on $\pfin{X}$, the finite power set of $X$.  There are actually multiple such
quasi-orders in common use, but we will only concern ourselves with one of them.

\begin{defn}
Given $S,T\in\pfin{X}$, we define $S\le_{\mathrm{m}} T$ if for each $s\in S$
there is some $t\in T$ such that $s\le t$.
\end{defn}

\begin{notn}
Since we are only using $\le_{\mathrm{m}}$, and are not using any of the other
orders on $\pfin{X}$, we will simply write $S\le T$ rather than
$S\le_{\mathrm{m}} T$.
\end{notn}

This is also a well-quasi-order, as follows from Higman's Lemma.  Once again we
can put a bound on its type:

\begin{thm}[Abriola et.~al.~\cite{Abriola}]
\label{pow2}
If $X$ is a well-quasi-order, then
\[ o(\pfin{X})\le 2^{o(X)}. \]
Moreover, this bound is tight; given an ordinal $\alpha$, there exists a well
partial order $H_\alpha$ with $o(H_\alpha)=\alpha$ and
$o(\pfin{H_\alpha})=2^\alpha$.
\end{thm}

For convenience, we will also make the following definition:
\begin{defn}
If $X$ is a set (particularly, a well-quasi-order), we define
\[ \pfinz{X} := \pfin{X} \setminus \{\emptyset\}. \]
\end{defn}

Observe that $o(\pfinz{X})=-1+o(\pfin{X})$; so, they are equal if $X$ is
infinite, and $o(\pfinz{X})$ is one smaller if $X$ is finite.

\section{Proof of the upper bound}
\label{secproof}

In this section we prove the upper bound.  As mentioned above, the proof is
essentially similar to Erd\H{o}s and Rado's earlier proof \cite{ER} that did not
provide an upper bound.  The proof proceeds by induction on $k$, starting from
the already-known case of $k=1$.

However, we make two alterations to Erd\H{o}s and Rado's approach.  The first
is is in how we approach sequences of length $\omega^k$; Erd\H{o}s and Rado's
proof decomposed these by considering them as a sequence of length
$\omega^{k-1}$ over an alphabet consisting of sequences of length $\omega$.  We
reverse this, and consider them as a sequence of length $\omega$ over an
alphabet consisting of sequences of length $\omega^{k-1}$.  I do not believe
this change makes any substantial difference; I simply found the proof easier to
carry out in this form.

The second change, however, is what allows us to derive a better bound than the
one implicit in their proof.  Where Erd\H{o}s and Rado simply iterated the
operation $X\mapsto X^*$, resulting in a tower of $2n$ exponentials, we will use
this operation only once, and rely on $\pfinn$ (or rather $\pfinnz$) for the
rest, resulting in a tower of only $n+1$ exponentials.

We start by examining, what do sequences of length $\omega$ look like?  This
question was answered by Erd\H{o}s and Rado, but we will repeat the proofs here.
Let us make a definition:

\begin{defn}
If $X$ is a well-quasi-order, let $\finseqeq{\alpha}{X}$ be the subset of
$\finseq{\alpha+1}{X}$ consisting of sequences of length exactly $\alpha$.
(Similarly, $\seqeq{\alpha}{X}$ will be the subset of $\seq{\alpha+1}{X}$
consisting of sequences of length exactly $\alpha$.)
\end{defn}

It's also useful here to know about indecomposable sequences:

\begin{defn}
A sequence over a well-quasi-order $X$ is said to be \emph{indecomposable}
\cite[Section 7.6.1]{Fra} if it is equivalent to all of its nonempty tails.
Note that the length of an indecomposable sequence must be an indecomposable
ordinal, that is, a power of $\omega$.  We will denote the set of
indecomposable finite-image sequences of length exactly $\alpha$ over an
alphabet $X$ by $\indec{\alpha}{X}$.
\end{defn}

Note that if a sequence is indecomposable, then so is any other sequence
equivalent to it, since if $s\equiv t$ and $s$ indecomposable, then given a tail
$t'$ of $t$, using $s\le t$ one may find a tail $s'$ of $s$ with $s'\le t'$, so
$t\le s\le s'\le t'$.

\begin{prop}
\label{images}
The map $\im:\finseq{\alpha}{X}\to\pfin{X}$ mapping each sequence to its image
is monotone.
\end{prop}

\begin{proof}
If $s,t\in\finseq{\alpha}{X}$ with $s\le t$, then any element of $\im(s)$
by definition occurs in $s$, say as $s_i$ for $i<\alpha$; then $s_i\le t_j$
for some $j<\alpha$, and $t_j\in\im(t)$, so $\im(s)\le\im(t)$.
\end{proof}

\begin{prop}[Erd\H{o}s, Rado \cite{ER}]
\label{repeat}
Let $X$ be a well-quasi-order.  Define the map $\varphi_X:\pfinz{X}\to
\indec{\omega}{X}$ as follows: To take $\varphi_X(S)$, write the elements of
$S$ in an arbitrary order, so $S=\{s_0,\ldots,s_{p-1}\}$.  Then define
$\varphi_X(S)=(s_0s_1\cdots s_{p-1})^\omega$. Then $\varphi_X$ is
well-defined up to equivalences and monotonic.
\end{prop}

\begin{proof}
First, observe that the outputs are all indecomposable, since any nonempty tail
of $(s_0s_1\cdots s_{p-1})^\omega$ is of the form $s_r \cdots
s_{p-1}(s_0s_1\cdots s_{p-1})^\omega$ for some $r$ and so contains
$(s_0s_1\cdots s_{p-1})^\omega$ in addition to being contained in it.

Now, we show monotonicity and well-definedness.
Suppose $S=\{s_0,\ldots,s_{p-1}\}$ and $T=\{t_0,\ldots,t_{q-1}\}$, and suppose
that for each $0\le i<p$, there is some $j(i)$ such that $s_i \le
t_{j(i)}$.  We wish to show that \[(s_0s_1\cdots s_{p-1})^\omega \le
(t_0t_1\cdots t_{q-1})^\omega.\]

Write
\[(t_0t_1\cdots t_{q-1})^\omega = ((t_0t_1\cdots t_{q-1})^p)^\omega.\]

Certainly
\[s_0 s_1\cdots s_{p-1} \le (t_0t_1\cdots t_{q-1})^p,\]
since each $s_i$ is less than or equal to the $i$'th copy of $t_{j(i)}$.
Extending this implies 
\[(s_0s_1\cdots s_{p-1})^\omega \le
((t_0t_1\cdots t_{q-1})^p)^\omega = (t_0t_1\cdots t_{q-1})^\omega,\]
as required.

Applying this with $S=T$ then shows that $\varphi_X$ is well-defined up to
equivalences, and applying it for arbitrary $S$ and $T$ shows that $\varphi_X$
is monotonic.
\end{proof}

\begin{prop}[Erd\H{o}s, Rado \cite{ER}]
\label{omega-surj}
Let $X$ be a well-quasi-order. Define the map $\varphi_X:X^*\times\pfinz{X}\to
\finseqeq{\omega}{X}$ by $\varphi_X(t,S)=t\varphi_X(S)$.  Then this $\varphi_X$
is well-defined up to equivalences; is monotonic; and is surjective up to
equivalences.  Moreover, given $s\in\finseqeq{\omega}{X}$, we may take $S$ to
consist of all elements of $X$ that occur infinitely often in $s$.
\end{prop}

You will notice here the reuse of the symbol $\varphi_X$; this is deliberate.
We will soon be defining many maps of the form $\varphi_X:
F(X)\to\finseq{\alpha}{X}$, each for a different $F(X)$.  We could denote these
$\varphi_{X,F(X)}$, but since there will be no ambiguity, we simply write
$\varphi_X$ for all of them, which will make things easier.

\begin{proof}
Since $\varphi_X:\pfinz{X}\to\finseqeq{\omega}{X}$ is well-defined up to
equivalences and monotonic, it follows immediately that
$\varphi_X:X^*\times\pfinz{X}\to\finseqeq{\omega}{X}$ is as well.  The problem,
then, is to show that it is surjective.

Suppose $s=s_0 s_1\cdots\in \finseqeq{\omega}{X}$.  Let $S=\{s_0,s_1,\ldots\}$;
then $S$ is a finite set.  Let $T\subseteq S$ be the set of elements of $S$ that
occur infinitely often in $S$.  Because $S\setminus T$ is finite, and each
element of it occurs only finitely often, there must be some index $r$ such
that, for all $i\ge r$, $s_i\in T$.  Let $s' = s_r s_{r+1}\cdots$.

We want to show, then, that $\varphi_X(s_0\cdots s_{r-1}, T)$ is equivalent to
$s$.  Since the initial $r$ elements are equal, it suffices to show that $s'$ is
quivalent to $\varphi_X(T)$.

Say $T=\{t_0,\ldots,t_{p-1}\}$.  Then for each $i\ge r$ we have $s_i \in T$ and
therefore $s_i \le t_0\cdots t_{p-1}$; this then implies \[s' \le (t_0\cdots
t_{p-1})^\omega = \varphi_X(T).\]

For the reverse, to show $\varphi_X(T)\le s'$, we can inductively define a
strictly increasing $j:\omega\to\omega$ such that
\[
((t_0\cdots t_{p-1})^\omega)_i \le s'_{j(i)}.
\]

Assume that we have defined $j(i')$ for all $i'<i$, and we want to define
$j(i)$.  Now, $\varphi_X(T)_i$ is some element of $T$, which means it occurs
infinitely often in $s'$.  Therefore, it occurs at least once with index greater
than $j(i')$ for any $i'<i$.  So, define $j(i)$ to be such an index.

Then 
\[
((t_0\cdots t_{p-1})^\omega)_i \le s'_{j(i)}
\]
as required, showing that $\varphi_X(T)\le s'$.  So the two are equivalent, and
$\varphi_X$ is surjective up to equivalences.  Moreover, $T$ was selected in the
manner required.
\end{proof}

\begin{cor}
\label{indec}
Let $X$ be a well-quasi-order.  Then $\varphi_X:\pfinz{X}\to \indec{\omega}{X}$
is surjective up to equivalences.  Moreover, if a finite-image length-$\omega$
sequence is indecomposable, with image $T$ and $S$ the set of elements that
occur infinitely often, then $T$ and $S$ are equivalent, so that $s$ is
indecomposable if and only if it is equivalent it $\varphi_X(T)$ and also if and
only if it is equivalent to $\varphi_X(S)$.
Finally, if each element of $s$ occurs infinitely often, $s$ is indecomposable.
\end{cor}

Note that the converse to the final statement is false, as the sequence
$01^\omega$ over $\{0,1\}$ is indecomposable, but $0$ occurs only once in it.

\begin{proof}
By Proposition~\ref{omega-surj}, if $s\in\indec{\omega}{X}$, then
$s=s'\varphi_X(S)$ for some $s'\in X^*$ and $S\in\pfinz{X}$ (more specifically,
$S$ is the set of elements that occur infinitely often in $s$).  But this means
that $\varphi_X(S)$ is a nonempty tail of $s$, and $s$ was assumed
indecomposable, proving the surjectivity claim.

Moreover, if $s$ is indecomposable, by Proposition~\ref{images}, the image of
$S$, which is $T$, is equivalent to the image of $\varphi_X(S)$, which is $S$.
Thus, in addition to it being the case that $s$ is equivalent
to $\varphi_X(S)$, $s$ is also equivalent to $\varphi_X(T)$.

Finally, if each element in a sequence of length $\omega$ occurs infinitely
often, then by Proposition~\ref{omega-surj} it is equivalent to $\varphi_X(S)$,
where $S$ is its image, and thus by Proposition~\ref{repeat} is indecomposable.
\end{proof}

So now we know what sequences of length $\omega$ look like; in order to get a
handle on sequences of length $\omega^k$, however, we're going to have to make
the domain of our mapping a bit bigger.

\begin{defn}
\label{defn-pkqk}
Let $X$ be a well-quasi-order.  Define $P_k(X)$ and $Q_k(X)$ inductively as
follows:
\begin{enumerate}
\item $P_0(X)=X$
\item $Q_k(X) = \coprod_{i=0}^{k-1} P_i(X)$
\item For $k>0$, $P_k(X) = \pfinz{Q_k(X)}$
\end{enumerate}
\end{defn}

Note that all of these are well-quasi-orders, and their types may be bounded (in
terms of $o(X)$) by application of Theorem~\ref{pow2}.  Note also that $Q_0(X)$
is the empty set.

Now, just as we mapped $\pfinz{X}$ and $X^*\times\pfinz{X}$ to
$\finseqeq{\omega}{X}$, we'll do the same with these to cover
$\finseq{\omega^k}{X}$ and $\finseqeq{\omega^k}{X}$.

\begin{defn}
\label{defn-phi}
Let $X$ be a well-quasi-order.  Define a map
$\varphi_X:P_k(X)\to\finseq{\omega^k+1}{X}$, and (for $k>0$)
a map $\varphi_X:Q_k(X)\to\finseq{\omega^{k-1}+1}{X}$, as follows.  (Again,
we'll continue using the same symbol for all of these.)
\begin{enumerate}
\item
$\varphi_X:P_0(X)\to\finseq{2}{X}$ will simply be the inclusion map $X\to
X\cup\{\varepsilon\}$ (where here on the right-hand side we are identifying
$X$ with sequences of length $1$ over $X$).
\item $\varphi_X:Q_k(X)\to\finseq{\omega^{k-1}+1}{X}$ will simply be the
disjoint union of $\varphi_X$ on each of the $P_i(X)$ ($i<k$) that make it up.
\item To define
$\varphi_X:P_k(X)\to\finseq{\omega^k+1}{X}$ for $k>0$, suppose $S\in
\pfinz{Q_k(X)} = \{s_0, s_1, \ldots, s_{p-1}\}$; then
\[ \varphi_X(S) = (\varphi_X(s_0)\cdots\varphi_X(s_{p-1}))^\omega. \]
\end{enumerate}
\end{defn}

Note how case (3) in Definitions~\ref{defn-pkqk} and
\ref{defn-phi} is connected to Propositions~\ref{repeat} and \ref{omega-surj};
we take finite sets to express sequences of the form $s^\omega$.  Thus, for
instance, an element $x\in X$ is quite different from $\{x\}\in\pfinz{X}$, as
$\varphi_X(x)=x$, while $\varphi_X(\{x\})=x^\omega$.

\begin{prop}
\label{prop-pkqk}
The map $\varphi_X:P_k(X)\to \finseq{\omega^k+1}{X}$ and, for $k>0$, the map
$\varphi_X:Q_k(X)\to \finseq{\omega^{k-1}+1}{X}$ are well-defined up to
equivalence and are monotonic.  Also, the image of the each map consists of
indecomposable sequences.
\end{prop}

\begin{proof}
We induct on $k$, alternating between $P$ and $Q$ in the order $P_0(X)$,
$Q_1(X)$, $P_1(X)$, $Q_2(X)$, etc.  Obviously for each $k$, if it is true for
$P_\ell(X)$ for all $\ell<k$ then it is true for $Q_k(X)$.  And clearly it is
true for $P_0(X)$.  That leaves the case of $P_k(X)$ with $k>0$.

So for $P_k(X)$ with $k>0$, assume the statement is true for $Q_k(X)$; let us
denote $Q_k(X)$ by $Y$ for convenience. By
Proposition~\ref{repeat}, the map
\[\varphi_Y:\pfinz{Q_k(X)}\to\indec{\omega}{Q_k(X)}\] is
monotonic and well-defined up to equivalences.
That is, given $S,T\in P_k(X)$,
we have $\varphi_Y(S) \le
\varphi_Y(T)$, where here we are comparing length-$\omega$ sequences
over $Q_k(X)$, using the order on $Q_{k-1}(X)$.

Now, if we take $\varphi_Y(S)$ and apply to each of its elements the
map $\varphi_X$, and concatenate all of these in place, we obtain
$\varphi_X(S)$.  But since $\varphi_X:Q_k(X)\to \finseq{\omega^{k-1}+1}{X}$
is monotonic,
and concatenation (including infinitary concatenation) is monotonic, that means
that overall we have described a monotonic operation, and $\varphi_X:P_k(X)\to
\finseq{\omega^k+1}{X}$ is monotonic.

Finally, we  must show indecomposability.  It suffices to handle $P_k(X)$, as
the image of $Q_k(X)$ is a union of images of $P_i(X)$.  If $k=0$, then the
image of $P_0(X)=X$ consists of sequences of length $1$, which are necessarily
indecomposable.  While if $k>0$, then each sequence in the image of $P_k(X)$ is
of the form $s^\omega$ for some sequence $s$, and is therefore also
indecomposable.
\end{proof}

\begin{defn}
Let $X$ be a well-quasi-order.  We define
$\varphi_X:Q_k(X)^*\to\finseq{\omega^k}{X}$ by
$\varphi_X(a_0\cdots a_n)=\varphi_X(a_0)\cdots\varphi_X(a_n)$.
\end{defn}

\begin{defn}
Let $X$ be a well-quasi-order.  We define $\varphi_X:Q_k(X)^* \times
P_k(X)\to\finseq{\omega^k+1}{X}$ by $\varphi_X(a,b)=\varphi_X(a)\varphi_X(b)$.
\end{defn}

\begin{thm}
\label{mainthm}
The maps $\varphi_X:Q_k(X)^*\to\finseq{\omega^k}{X}$, $\varphi_X:Q_k(X)^*
\times P_k(X)\to\finseq{\omega^k+1}{X}$ are well-defined up to equivalence and
are monotonic.  In addition, the former map is surjective up to equivalences,
and the image \[\varphi_X(Q_k(X)^*\times P_k(X))\] contains
$\finseqeq{\omega^k}{X}$ up to equivalences.  Finally, the image
$\varphi_X(P_k(X))$ contains $\indec{\omega^k}{X}$ up to equivalences.
\end{thm}

Note that in the case $k=1$, this is covered by Propositions~\ref{repeat} and
\ref{omega-surj}.

\begin{proof}
That these maps are monotonic and well-defined up to equivalences follows from
Proposition~\ref{prop-pkqk}.  The problem, then, is to prove the statements
about their images.

In fact, it suffices to prove these statements under the assumption that $X$ is
finite.  To see this, observe that since any $s$ in either
$\finseq{\omega^k}{X}$ or $\finseqeq{\omega^k}{X}$ uses only finitely many
elements from $X$, then if we let $X'\subseteq X$ be this finite image of $s$,
to prove the statement for $s$ it suffices to prove the statement with $X'$ in
place of $X$.  So, if we prove it for all finite $X'$ and $s$, we prove it for
all $s$, and therefore for all $X$ and $s$, i.e., we have proved it for all $X$
regardless of finiteness.  As such, for the rest of this proof, we assume $X$ is
finite.

We will induct on $k$; assume the statements are true for all $\ell<k$.
Then for $Q_k(X)^*$, say we are given $s\in \finseq{\omega^k}{X}$.  If
$s=\varepsilon$, then $s=\varphi_X(\varepsilon)$, so assume $s\ne\varepsilon$.
Let $\alpha=\dom(s)<\omega^k$, and write $\alpha$ in Cantor normal form as
$\alpha=\omega^{\ell_0}+\ldots+\omega^{\ell_r}$, where $\ell_0\ge\ldots\ge
\ell_r$.  Then we may correspondingly split up $s$ as $s_0\cdots s_r$ with
$\dom(s_i)=\omega^{\ell_i}$.  Since each $\ell_i<k$, we may apply the inductive
hypothesis to conclude that each $s_i$ is (up to equivalences) in the image of
the map
\[ \varphi_X: Q_{\ell_i}(X)^*\times P_{\ell_i}(X)\to\finseq{\omega^{\ell_i}+1}{X}.\]
If we write $s_i = \varphi_X(t_i, t'_i)$, then we may define
\[ t = t_0 t'_0 \cdots t_r t'_r; \]
since, for $\ell<k$, we have $Q_\ell(X)\subseteq Q_k(X)$ and $P_\ell(X)\subseteq
Q_k(X)$, we have $t\in Q_k(X)^*$.  Moreover, $\varphi_X(t)=s$, proving the
statement for $Q_k(X)^*$.

For $Q_k(X)^* \times P_k(X)$, say now we are given $s\in\finseqeq{\omega^k}{X}$.
If $k=0$, then $s$ consists of just a single element $x\in X$, and so we can
write $s=\varphi_X(\varepsilon, x)$.

If $k>0$, then we may decompose $s$ as $s=s_0s_1s_2\cdots$, where
$\dom(s_i)=\omega^{k-1}$.  We now apply the inductive hypothesis to write $s_i =
\varphi_X(w_i, t_i)$ (up to equivalences) where $w_i\in Q_{k-1}(X)^*$ and
$t_i\in P_{k-1}(X)$.
For convenience, define
\[Y=\pfinz{Q_{k-1}(X)}\times P_{k-1}(X)=(P_{k-1}(X))^2.\]

So consider the sequence 
\[ v := (\im(w_i), t_i)_{i<\omega} \]
of length $\omega$ over the alphabet $Y$, where here $\im(w)$ denotes the image
of a sequence $w$.  We wish to apply
Proposition~\ref{omega-surj} to $v$, but first we need to know that $v$ uses
only finitely many elements of $Y$.  This is where we apply our assumption that
$X$ is finite.  Since $X$ is finite, we see by induction that all $P_i(X)$ and
$Q_i(X)$ are also finite, since each $Q_i(X)$ is a finite union of finite sets
and each $P_i$ is a subset of the power set of a finite set.  Therefore, $Y$ is
also finite, and as such, $v$ can only use finitely many elements from it.

With that settled, we may now apply Proposition~\ref{omega-surj} to $v$ to
obtain $u\in Y^*$ and $S\in \pfinz{Y}$ such that $\varphi_Y(u,S) = v$ (up to
equivalences).

This means that, for some $p$,
\[u=(\im(w_0),t_0), (\im(w_1),t_1), \ldots, (\im(w_{p-1}),t_{p-1}).\]
Recall here that $w_i\in Q_{k-1}(X)^*$ and $t_i\in P_{k-1}(X)$, and
\[ Q_{k-1}(X) \cup P_{k-1}(X) \subseteq Q_k(X), \]
so if we write
\[u'=w_0t_0\cdots w_{p-1} t_{p-1},\]
then $u'\in Q_k(X)^*$.  Note that $\varphi_X(u')=s_0\cdots s_{p-1}$.

Now we consider $S$, which we will write as
$S=\{(T_0,S_0),\ldots,(T_{q-1},S_{q-1})\}$.
Let $T$ be the union of the $T_i$, and let 
\[U=T\cup\{S_0,\ldots,S_{q-1}\}\in\pfinz{Q_k(X)}.\]
For convenience once again, define
$Z=Q_k(X)$.

Let $s'=s_p s_{p+1}\cdots$, and
let \[t=w_p t_p w_{p+1} t_{p+1} \cdots\in \seqeq{\omega}{Z};\]
we wish to show
that, up to equivalence,
$\varphi_Z(U) = t$.
From above, we know that, up to equivalence,
\[
t_p t_{p+1} \cdots = \varphi_Z(\{S_0,\ldots,S_{q-1}\})
= (S_0 S_1 \cdots S_{q-1})^\omega
\]
and that the images of the $w_i$ also (up to equivalence) form a repeating sequence $T_0, T_1, \ldots, T_{q-1},
T_0, T_1$, \ldots; recall here that each such image is an element of the set
$\pfinz{Q_{k-1}(X)}\subseteq\pfinz{Z}$.
So, one may apply appropriate equivalences to $t$ to find an equivalent sequence
$t'$ with $\im(t')=U$ (a finite set).

Moreover, each element of $U$ occurs infinitely often in $t$ (a length-$\omega$ sequence over $Z$); each
$S_i$ occurs infinitely often as one of the $t_i$, and each element of $T$ occurs
infinitely often among the $w_i$ (since it occurs in some $\im(w_i)$, and there
are infinitely many $w_i$ with that same image).
Therefore, $\varphi_Z(U)\le t$ by the same reasoning as in
Proposition\ref{repeat} (or by an application of monotonicity), and $t\equiv
t'\le \varphi_Z(U)$ since $\im(t')=U$, so $t\equiv\varphi_Z(U)$, and therefore
is indecomposable by Proposition~\ref{repeat}.

So if we now pass from $Z$ down to $X$, we conclude by the monotonicity of
concatenation that $s'$, which recall is
given by
\[
s' = s_p s_{p+1} \cdots = \varphi_X(w_p) \varphi_X(t_p) \varphi_X(w_{p+1})
\varphi_X(t_{p+1}) \cdots,
\]
is equivalent to $\varphi_X(U)$;
and since we have shown that $\varphi_X(u')=s_0\cdots s_{p-1}$,
we conclude that $\varphi_X(u',U)=s$.  Since $u'\in Q_k(X)^*$ and $U\in P_k(X)$,
that proves the required surjectivity in this case.

Finally, we handle the case of indecomposable sequences.  Say we are given
$s\in\indec{\omega^k}{X}$.  By the preceding argument, one may write (up to
equivalences) $s=\varphi_X(u')\varphi_X(S)$ for some $(u',S)\in Q_k(X)^* \times
P_k(X)$.  But by assumption, $s$ is indecomposable, and $\varphi_X(S)$ is a
nonempty tail of it; so $s$ is equivalent to $\varphi_X(S)$, concluding the
proof of the theorem.
\end{proof}

\begin{rem}
Based on Proposition~\ref{omega-surj}, which shows that every length-$\omega$
finite-image sequence has an indecomposable tail, one might wonder whether
the same is true for more general finite-image sequences; in fact, this is true
(see \cite[Section 7.6.4]{Fra}), but we will not show it here.
\end{rem}

We can then write out explicitly the bound this gets us.

\begin{defns}
Define sequences of functions $p_k$ and $q_k$ recursively via:
\begin{enumerate}
\item $p_0(\beta)=\beta$
\item $q_k(\beta) = \bigoplus_{i=0}^{k-1} p_i(\beta)$
\item For $k>0$, $p_k(\beta) = -1 + 2^{q_k(\beta)}$.
\end{enumerate}
\end{defns}

\begin{thm}
\label{upper}
Let $h$ be the function such that $o(X^*)=h(o(X))$ (see
Theorem~\ref{higman-type}).  Then, for any well-quasi-order $X$ and any whole
number $k$, one has
\[
o(\finseq{\omega^k}{X}) \le h(q_k(o(X))),
\]
\[
o(\finseqeq{\omega^k}{X}) \le h(q_k(o(X)))\otimes p_k(o(X)),
\]
and
\[
o(\indec{\omega^k}{X}) \le p_k(o(X)).
\]
\end{thm}

This is the more explicit version of Theorem~\ref{upper-intro}.

\begin{proof}
This follows immediately from Theorem~\ref{mainthm} together with
Theorem~\ref{higman-type} as well as the other preliminaries mentioned in
Section~\ref{prelim}.
\end{proof}

\begin{exmp}
Let's examine what this gets us for $\finseq{\omega^k}{x}$ for small $k$.  For
the trivial case of $k=0$, we get $o(\finseq{1}{X})\le 1$, as we should,
and for the well-known case of $k=1$, we get
\[ o(\finseq{\omega}{X})\le h(o(X)) \le \omega^{\omega^{o(X)+1}}, \]
again as we should.

Meanwhile, for our first new case, $k=2$, we get
\[
o(\finseq{\omega^2}{X}) \le
h((-1 + 2^{o(X)}) \oplus o(X)) \le
\omega^{\omega^{(-1 + 2^{o(X)}) \oplus o(X) + 1}}.
\]

Let's now consider what happens if $o(X)=2$ (that is, $|X|=2$).  In this case,
the bound becomes
\[ 
o(\finseq{\omega^2}{X}) \le h(5) = \omega^{\omega^4}.
\]
Now, if $|X|=2$, then either $X$ is a chain of size $2$, or an antichain of size
$2$.  In the case of a chain, say $X=\{0,1\}$, this is certainly an
overestimate; the actual type would be at most $h(4)=\omega^{\omega^3}$, as
$(01)^\omega$ and $1^\omega$ are equivalent, so one has only $4$ distinct
building blocks, not $5$.  In the case of an antichain, by contrast, it is
possible that this bound is tight, although this author has not verified
whether it is.
\end{exmp}

\subsection{The more general case of $\alpha<\omega^\omega$}
\label{cnf}

Let us briefly address the more general case of $\finseq{\alpha}{X}$ and
$\finseqeq{\alpha}{X}$ for $\alpha<\omega$.  Of course, if one only wants crude
bounds, one can simply find some $k$ such that $\alpha\le\omega^k$ and then
apply Theorem~\ref{upper}, but it's not hard to do better, so let's address
that.

For convenience, let's make some definitions first.

\begin{defns}
As above, we'll define $h$ to be the function such that $o(X^*)=h(o(X))$ (see
Theorem~\ref{higman-type}).  We'll also define
\[ f_k(\beta) = h(q_k(\beta)) \]
and
\[ g_k(\beta) = h(q_k(\beta))\otimes p_k(\beta), \]
so that (per Theorem~\ref{upper}) we have $o(\finseq{\omega^k}{X})\le f_k(o(X))$ and $o(\finseqeq{\omega^k}{X})\le
g_k(o(X))$.
\end{defns}

With this, the case of $\finseqeq{\alpha}{X}$ is quite simple:

\begin{thm}
\label{finseqeq-gen}
Suppose $\alpha=\omega^{k_0} + \ldots + \omega^{k_r}$, where $\omega > k_0 \ge
\ldots \ge k_r$.
Then
\[ o(\finseqeq{\alpha}{X}) \le
\bigotimes_{i=0}^r g_{k_i}(o(X)).
\]
\end{thm}

\begin{proof}
Concatenation forms a monotonic surjection from $\prod_{i=0}^r
\finseqeq{\omega^{k_i}}{X}$ to $\finseqeq{\alpha}{X}$.
\end{proof}

We can handle the case of $\finseq{\alpha}{X}$ with only slightly more
complication.

\begin{thm}
\label{finseq-gen}
Suppose $\alpha=\omega^{k_0} + \ldots + \omega^{k_r}$, where $\omega > k_0 \ge
\ldots \ge k_r$.
Then
\[ o(\finseq{\alpha}{X}) \le
\bigoplus_{i=0}^r \left( \left(\bigotimes_{j=0}^{i-1}
g_{k_j}(o(X))\right)\otimes  f_{k_i}(o(X))\right) .
\]
\end{thm}

\begin{proof}
For any sequence of length less than $\alpha$, we can find the smallest $i$ so
that its length is less than $\omega^{k_0}+\ldots+\omega^{k_i}$; then the
sequence breaks down into sequences of length equal to $\omega^{k_j}$ (for
$j<i$) and a sequence of length less than $\omega^{k_i}$.  This gives us a
monotonic surjection onto
$\finseq{\alpha}{X}$ from
\[
\coprod_{i=0}^r \left( \left(\prod_{j=0}^{i-1}
\finseqeq{\omega^{k_j}}{X}\right)\times  \finseq{\omega^{k_i}}{X}\right) .
\]
\end{proof}

However, when there's a nonzero finite part it's actually possible to do a
little better than this.

\begin{defns}
Define $P_k^0(X)=\pfin{Q_k(X)}$, i.e., $P_k^0(X)=P_k(X)\cup\{\emptyset\}$.  Then
extend $\varphi_X$ from $P_k(X)$ to $P_k^0(X)$ via
$\varphi_X(\emptyset)=\varepsilon$.  We similarly extend $\varphi_X$ from
$Q_k(X)^*\times P_k(X)$ to $Q_k(X)^*\times P_k^0(X)$ in the obvious manner.
\end{defns}

\begin{prop}
\label{mainthmplus}
The map $\varphi_X:P_k^0(X)\to\finseq{\omega^k+1}{X}$ is well-defined up to
equivalences.  The map $\varphi_X:Q_k(X)^*\times P_k^0(X) \to
\finseq{\omega^k+1}{X}$ is well-defined up to equivalences, monotonic, and
surjective up to equivalences.
\end{prop}

\begin{proof}
Since we already know that $\varphi_X:P_k(X)\to\finseq{\omega^k+1}{X}$ is
well-defined up to equivalences and monotonic by Proposition~\ref{prop-pkqk},
the same obviously also holds on $P_k^0(X)$, which then implies it also holds
(by Proposition~\ref{prop-pkqk} again) on $Q_k(X)^*\times P_k^0(X)$.  So the
only question is surjectivity.

Any element of $\finseq{\omega^k+1}{X}$ is either an element of
$\finseq{\omega^k}{X}$ or $\finseqeq{\omega^k}{X}$.  In the former case, by
Theorem~\ref{mainthm}, it (up to equivalences) lies in the image of $Q_k(X)^*$,
i.e., in the image of $Q_k(X)^*\times\{\emptyset\}$.  In the latter case, again
by Theorem~\ref{mainthm}, it (up to equivalences) lies in the image of
$Q_k(X)^*\times P_k(X)$.  Either way, up to equivalences, it lies in the image
of $Q_k(X)^*\times P_k^0(X)$.
\end{proof}

Based on this, we define:

\begin{defn}
Let $p^+_k(\beta)=1+p_k(\beta)$ and $g^+_k(\beta)=f_k(\beta)\otimes
p^+_k(\beta)$.
\end{defn}

And from this we conclude the following proposition and theorem:

\begin{prop}
\label{plusone}
If $X$ is a well-quasi-order, one has
$o(\finseq{\omega^k+1}{X})\le g^+_k(o(X))$.
\end{prop}

\begin{proof}
This follows immediately from Proposition~\ref{mainthmplus}.
\end{proof}

\begin{thm}
\label{plusone-gen}
Suppose $\alpha=\omega^{k_0} + \ldots + \omega^{k_r} + \ell$, where $\omega >
k_0 \ge \ldots \ge k_r > 0$ and $1\le\ell<\omega$.

Then
\begin{multline*}
o(\finseq{\alpha}{X}) \le
\bigoplus_{i=0}^{r-1} \left( \left(\bigotimes_{j=0}^{i-1}
g_{k_j}(o(X))\right)\otimes f_{k_i}(o(X)) \right)
\oplus\\
\left(\left(\bigotimes_{j=0}^{r-1}
g_{k_j}(o(X))\right)\otimes g^+_{k_r}(o(X)) \right)
\oplus
\left(\left(\bigotimes_{j=0}^r
g_{k_j}(o(X))\right)\otimes \left(\bigoplus_{t=1}^{\ell-1}o(X)^{\otimes
t}\right)\right)
.
\end{multline*}
\end{thm}

\begin{rem}
Note that the final term in this sum is not actually unique to this theorem;
if one were to apply Theorem~\ref{finseq-gen} in this case and expand it out,
one would see it there too (since $f_0(\beta)=1$ and $g_0(\beta)=\beta)$). It is
only the second-to-last term, that involves $g_{k_r}^+$, that is different.
\end{rem}

\begin{proof}
For any sequence of length less than $\alpha$, either its length is less than
$\omega^{k_0}+\ldots+\omega^{k_{r-1}}$, or it is at least that long.  In the
former case, we apply Theorem~\ref{finseq-gen}.  In the latter case, we can
break it down into a sequence of length $\omega^{k_0}+\ldots+\omega^{k_{r-1}}$
(to which Theorem~\ref{finseqeq-gen} applies) and a remainder of length less
than $\omega^{k_r}+\ell$, i.e., at most $\omega^{k_r}+(\ell-1)$.

We can then break this down into the case where the length of the remainder is
less than or equal to $\omega^{k_r}$, and the case where it is not, i.e., where
it is equal to $\omega^{k_r}+t$ for some $1\le t<\ell-1$.

Then the first case gives us the first part of the sum, the second case gives us
the second part of the sum, and the third case gives us the third part of the
sum.
\end{proof}

\begin{exmp}
Let's consider what this gets us when $\alpha=\omega^2+1$ and $o(X)=\omega$.
If we apply Theorem~\ref{finseq-gen}, we see that
\[o(\finseq{\omega^2+1}{X})\le
f_2(\omega) \oplus (f_0(\omega)\otimes g_2(\omega)) =
\omega^{\omega^{\omega2}+2}+\omega^{\omega^{\omega 2}}. \]
However, we apply Theorem~\ref{plusone}, we see that in fact
\[o(\finseq{\omega^2+1}{X})\le g^+_2(\omega) = \omega^{\omega^{\omega2}+2}. \]
\end{exmp}

It's probably possible to improve on this further, but we will stop here.

\section{Lower bounds}
\label{seclower}

We now turn to the question of lower bounds.  We'll prove our lower bounds on
$\finseq{\omega^2}{X}$ by embedding another ordering into it.

\begin{prop}
\label{genlower}
Suppose $X$ is a well partial order.  Then
$\varphi_X:(\pfinnz)^k(X)\to\indec{\omega^k}{X}$ is an embedding, and so is
$\varphi_X:(\pfinnz)^k(X)^*\to\finseq{\omega^{k+1}}{X}$.
\end{prop}

\begin{proof}
We induct on $k$.  Obviously, if $k=0$, then both maps are embeddings, as they
are simply identity maps.

Now, we will show that if the first part holds for $k$, then so does the second
part.  Say $s,t\in(\pfinnz)^k(X)^*$, and suppose $\varphi_X(s)\le\varphi_X(t)$,
with $s$ of length $m$ and $t$ of length $n$.  Now, each element of $s$ and of
$t$ is an element of $(\pfinnz)^k(X)$, so, under $\varphi_X$, they each  map to
a string of length $\omega^k$, which means that $\varphi_X(s)$ and
$\varphi_X(t)$ have length $\omega^k m$ and $\omega^k n$, respectively. More
specifically, for $a<m$, $b<n$, and $\alpha,\beta<\omega^k$, one has
\[ \varphi_X(s)_{\omega^k a+\alpha} = \varphi_X(s_a)_\alpha \]
and
\[ \varphi_X(t)_{\omega^k b+\beta} = \varphi_X(t_b)_\beta. \]
Now, that
$\varphi_X(s)\le\varphi_X(t)$ means that there exists strictly increasing
$i:\omega^k m\to\omega^k n$ such that $\varphi_X(s)_\zeta \le
\varphi_X(t)_{i(\zeta)}$ for all $\zeta<\omega^k m$.  We wish to construct
$j:m\to n$ such that $s_a \le t_{j(a)}$ for all $a<m$.

So, given $a<m$, look at the image of $i(\{\omega^k a +
\alpha:\alpha<\omega^k\})$.  There is some largest $b$ such that $\{\omega^k b +
\beta:\beta<\omega^k\}$ overlaps $i(\{\omega^k a + \alpha:\alpha<\omega^k\})$;
define $j(a)=b$.  Because $i$ is strictly increasing, for $a<m-1$, one has, for
$\alpha<\omega^k$, \[i(\omega^k a+\alpha)<i(\omega^k a+\alpha+1)\le
\sup_{\alpha<\omega^k} i(\omega^k a+\alpha) \le i(\omega^k (a+1)).\]  And again
since $i$ is strictly increasing, $\sup_{\alpha<\omega^k} i(\omega^k a+\alpha)$
must be an ordinal of Cantor-Bendixson rank at least $k$, i.e., a right-multiple
of $\omega^k$, which here means it is of the form $\omega^k c$ for some $c<n$.
Then $i(\omega^k a+\alpha)<\omega c$ for all $\alpha<\omega^k$, i.e., $j(a)<c$,
whereas $i(\omega(a+1)+\alpha)\ge i(\omega(a+1))\ge \omega c$, i.e., $j(a+1)\ge
c$.  So $j(a)<j(a+1)$ for all $a<m-1$, meaning $j$ is strictly increasing.

Now we must show that $s_a \le t_{j(a)}$.  Say $j(a)=b$.  There is
some smallest $\alpha<\omega^k$ such that $i(\omega^k a+\alpha)\ge \omega b$;
call it $\alpha_0$.
Let $w$ be the length-$\omega^k$ sequence given by
$w_\zeta=\varphi_X(s)_{\omega^k a+\alpha_0+\zeta}$.  Then $w\le
\varphi_X(t_b)$, because $i$ maps $w$ to a subsequence of the latter.  Moreover,
$w$ is a tail of $\varphi_X(s_a)$, and the latter is irreducible, so
$\varphi_X(s_a)\le \varphi_X(t_b)$.  By the hypothesis that the first part of
the proposition holds for $k$, we conclude $s_a \le t_b=t_{j(a)}$; since this
holds for all $a<m$, we conclude $s\le t$.

Now, we must show that if the first part of the proposition holds for $k$, it
holds for $k+1$.  Suppose $S,T\in(\pfinnz)^{k+1}(X)$; write
$S=\{S_0,\ldots,S_{p-1}\}$ and $T=\{T_0,\ldots,T_{q-1}\}$, so
\[\varphi_X(S)=(\varphi_X(S_0)\cdots\varphi_X(S_{p-1}))^\omega,\] and similarly
with $T$.   Additionally suppose
$\varphi_X(S)\le\varphi_X(T)$; this means once again we have strictly
increasing $i:\omega^{k+1} \to\omega^{k+1}$ such that $\varphi_X(S)_\zeta \le
\varphi_X(S)_{i(\zeta)}$ for all $\zeta<\omega^{k+1}$.  Once again we will
construct a corresponding $j:\omega\to\omega$ which is also strictly increasing.

To define $j(a)$, we once again look at the image of $i(\{\omega^k a +
\alpha:\alpha<\omega\})$; we will once again define $j(a)$ to be the largest $b$
such that such that $\{\omega^k b + \beta:\beta<\omega\}$ overlaps $i(\{\omega^k
a + \alpha:\alpha<\omega^k\})$.  However, in this case, we must first
demonstrate that there is such a $b$.  But there is, because the set
$i(\{\omega^k a + \alpha:\alpha<\omega\})$ is bounded above by $i(\omega^k
(a+1))$; it is not cofinal in $\omega^{k+1}$, and so $j(a)$ exists.

One may then use the same reasoning as in the previous case to prove that $j(a)$
is strictly increasing.  Also by the same reasoning, we can conclude that
$\varphi_X(S_a) \le \varphi_X(T_{j(a)\bmod q})$.  Therefore, applying the
inductive hypothesis, we conclude that $S_a\le T_{j(a)\bmod q}$; since this
holds for all $a$, we conclude $S\le T$, as required.
\end{proof}

Now, we'll need a family of well partial orders $H_\beta$ such that
$o(H_\beta)=\beta$ and such that we get large enough types upon performing other
operations on $H_\beta$.  Fortunately, \cite{Abriola} provides us with one.

\begin{defn}
Following Abriola et.~al.~\cite{Abriola}, define the well partial orders
$H_\beta$ recursively by
\begin{enumerate}
\item $H_1 = 1$
\item $H_\omega = \sum_{k<\omega} \coprod_{i<k} 1$
\item For $\gamma>0$, $H_{\omega^{\omega^\gamma}} =
\sum_{\delta<\omega^\gamma} H_{\omega^{\delta}}$
\item For $H_{\omega^\gamma}$ where $\gamma$ is not a power of $\omega$, write
$\gamma=\omega^{\delta_0}+\ldots+\omega^{\delta_r}$ ($\delta_i$ weakly
decreasing); then $H_{\omega^\gamma} = \prod_i H_{\omega^{\delta_i}}$, where the
product is a lexicographic product.
\item If $\beta$ is not a power of $\omega$, write
$\beta=\omega^{\gamma_0}+\ldots+\omega^{\gamma_r}$ ($\gamma_i$ weakly
decreasing); then $H_\beta = \coprod_i H_{\omega^{\gamma_i}}$.
\end{enumerate}
\end{defn}

Here the sums are ordered sums, i.e., concatenation.

\begin{rem}
There are any number of other similar ways one could construct a family
$K_\alpha$ with the same properties as $H_\alpha$.  For instance, one could
define
\begin{enumerate}
\item $K_1 = 1$
\item $K_{\omega^{\gamma+1}} = \sum_{k<\omega} K_{\omega^\gamma k}$
\item $K_{\omega^\gamma} = \sum_{\gamma'<\gamma} K_{\omega^{\gamma'}}$ if
$\gamma$ is a limit ordinal
\item If $\beta$ is not a power of $\omega$, write
$\beta=\omega^{\gamma_0}+\ldots+\omega^{\gamma_r}$ ($\gamma_i$ weakly
decreasing); then $K_\beta = \coprod_i K_{\omega^{\gamma_i}}$.
\end{enumerate}
This is not always isomorphic, as can be seen by taking
$\alpha=\omega^{\omega^2+\omega}$ ($H_{\omega^{\omega^2+\omega}}$ has two
elements of height $\omega^{\omega^2}$, whereas $K_{\omega^{\omega^2+\omega}}$
has only one), but it would work just as well.  We will stick to the family used
by Abriola et.~al., however.
\end{rem}

Abriola et.~al.~showed:

\begin{prop}[{\cite[Proposition~4.2 and Theorem~4.3]{Abriola}}]
\label{haworks}
For any $\beta$, one has $o(H_\beta)=\beta$ and $o(\pfin{H_\beta})=2^\beta$.
\end{prop}

We'll also need one other property of this family:

\begin{prop}
\label{hembed}
If $\beta\le\beta'$, then $H_\beta$ embeds in $H_{\beta'}$.
\end{prop}

\begin{proof}
We induct on $\beta'$.  Write $\beta$ in Cantor normal form as
$\gamma=\omega^{\gamma_0} + \ldots + \omega^{\gamma_r}$ with $\gamma_i$ weakly
decreasing, and similarly write $\beta'=\omega^{\gamma'_0} + \ldots +
\omega^{\gamma'_s}$.

Since $\beta\le\beta'$, either $\gamma=\gamma'_i$ for all $i$ and we have $r\le
s$, or there must be some smallest $k$ such that $\gamma_k < \gamma'_k$.  In
the first case, the embedding is immediate by the construction.

In the second case, we can use $H_{\omega^{\gamma_i}} = H_{\omega^{\gamma'_i}}$
for $i<k$; the problem then is to embed $\coprod_{i\ge k} H_{\omega^{\gamma_i}}$
inside $H_{\omega^{\gamma'_k}}$.

Since we are in the case where $\gamma_k < \gamma'_k$, we cannot have
$\gamma'_k=0$.  If $\gamma'_k$ is a power of $\omega$, then the embedding is
immediate by the construction.  If $\gamma'_k=\delta+1$, then by the inductive
hypothesis we may embed each $\gamma_i$ for $i>k$ into $H_\delta$; we can then
embed the whole into $H_{\omega^{\gamma'_k}}=H_{\omega^\delta}\cdot H_\omega$
since $H_\omega$ is infinite.

This leaves the case where $\gamma'_k$ is a limit ordinal but not a power of
$\omega$.  In this case, we may choose some successor ordinal $\gamma''$ with
$\gamma_k\le\gamma''<\gamma'_k$; then the whole embeds into
$H_{\omega^{\gamma''}}$ by the above, and the problem is now to embed
$H_{\omega^{\gamma''}}$ into $H_{\omega^{\gamma'_k}}$.

So expand $\gamma'_k$ into Cantor normal form as the sum of
$\omega^{\delta'_i}$, and $\gamma''$ into Cantor normal form as the sum of
$\omega^{\delta''_i}$, so that
\begin{eqnarray*}
H_{\omega^{\gamma'_k}}=\prod_{i=0}^{t'} H_{\omega^{\omega^{\delta'_i}}}
&
\textrm{and}
&
H_{\omega^{\gamma''_k}}=\prod_{i=0}^{t''} H_{\omega^{\omega^{\delta''_i}}}.
\end{eqnarray*}
Let $j$ be the smallest $i$ such that $\delta'_i\ne \delta''_i$; if there is no
such $i$, then we have $H_{\omega^{\gamma'_k}} = H_{\omega^{\gamma''}}$ which is
sufficient.  If there is such a $j$, then since $\gamma''\le\gamma''_k$, we have
$\delta''_j<\delta'_j$, and indeed $\delta''_i<\delta'_j$ for all $i\ge j$, and
so $\sum_{i=j}^{t''}\omega^{\delta''}<\omega^{\delta'_j}$.  Applying the earlier
case where the larger of the two ordinals is of the form
$\omega^{\omega^\zeta}$, we then obtain an embedding of
\begin{eqnarray*}
H_{\sum_{i=j}^{t''}\omega^{\omega^{\delta''}}}
& \textrm{into} &
H_{\omega^{\omega^{\delta'_j}}}
\end{eqnarray*}
and therefore of
\begin{eqnarray*}
H_{\sum_{i=j}^{t''}\omega^{\omega^{\delta''}}}
& \textrm{into} &
\prod_{i=j}^{t'} H_{\omega^{\omega^{\delta'_i}}}
\end{eqnarray*}
and thus finally of
\begin{eqnarray*}
H_{\omega^{\gamma''}} =
\left(\prod_{i=0}^{j-1} H_{\omega^{\omega^{\delta''_i}}}\right)
\cdot
H_{\sum_{i=j}^{t''}\omega^{\omega^{\delta''}}}
& \textrm{into} &
\prod_{i=0}^{t'} H_{\omega^{\omega^{\delta'_i}}}=H_{\omega^{\gamma'_k}}.
\end{eqnarray*}
This completes the proof.
\end{proof}

Putting all this together yields the following theorem:

\begin{thm}
\label{lowerbd}
For any ordinal $\beta$, there exists a well partial order $X$ with $o(X)=\beta$
and $o(\finseq{\omega^2}{X})\ge h(-1+2^\beta)$.  Specifically, we may take
$X=H_\beta$.
\end{thm}

\begin{proof}
By Proposition~\ref{genlower}, there is an embedding of $\pfinz{H_{\beta}}^*$
into $\finseq{\omega^2}{H_\beta}$; since $o(\pfinz{H(\beta)})=-1+2^\beta$,
$o(\pfinz{H_{\beta}}^*)=h(-1+2^\beta)$, and so we get
$o(\finseq{\omega^2}{H_\beta})\ge h(-1+2^\beta)$.
\end{proof}

Finally, we also note:

\begin{prop}
\label{lowerbd-indec}
Given an ordinal $\beta>0$, there exists a well partial order $X$ with
$o(X)=\beta$ and $o(\indec{\omega}{X})\ge 2^\beta$.
\end{prop}

\begin{proof}
This similarly follows from Propositions~\ref{genlower} and \ref{haworks}.
\end{proof}

We hope that Theorem~\ref{lowerbd} and Proposition~\ref{lowerbd-indec} can be
extended to the case of $\omega^k$ in the future.  Note that the limiting step
here is getting lower bounds on the type of iterations of $\pfinn$ beyond the
first.

\subsection*{Acknowledgements} Thanks to A.~D.~Chopra and to an anonymous
referee for various helpful comments.

\end{document}